\documentclass[twoside,11pt]{article}
\usepackage{graphicx,amsmath,amsthm,amssymb,latexsym,amsfonts,color,cite}
\usepackage{cases}
\usepackage{hyperref}

\newtheorem{thm}{\bf Theorem}

\newtheorem{example}{\bf Example}

\newtheorem{remark}{\bf Remark}

\begin{document}
\title{\bf A tighter $Z$-eigenvalue localization set for tensors and its applications}
\author{{Jianxing Zhao\footnote{Corresponding author. E-mail: zjx810204@163.com; zhaojianxing@gzmu.edu.cn (Jianxing Zhao)}}\\ [2mm]
{\small \textit{College of Data Science and Information Engineering, Guizhou Minzu University,}}\\
{\small \textit{Guiyang 550025, P.R. China}}\\
}

\date{}
\maketitle
{\bf Abstract.}
A new $Z$-eigenvalue localization set for tensors is given and proved to be tighter than those presented by 
Wang \emph{et al}. (Discrete and Continuous Dynamical Systems Series B 22(1): 187-198, 2017) and Zhao (J. 
Inequal. Appl., to appear, 2017). As an application, a sharper upper bound for the $Z$-spectral radius of weakly 
symmetric nonnegative tensors is obtained. Finally, numerical examples are given to verify the theoretical 
results.  \\[-3mm]

\thispagestyle{empty}

\section{Introduction}
For a positive integer $n$,~$n\geq 2$, $N$ denotes the set $\{1,2,\cdots,n\}$.
$\mathbb{C}$ ($\mathbb{R}$) denotes the set of all complex (real) numbers.
We call $\mathcal{A}=(a_{i_1i_2\cdots i_{m}})$ a real tensor of order $m$ dimension $n$, denoted by $\mathbb{R}^{[m,n]}$, if
$$a_{i_{1}i_2\cdots i_{m}}\in{\mathbb{R}},$$ where $i_{j}\in{N}$ for $j=1,2,\cdots,m$.
$\mathcal{A}$ is called nonnegative if $a_{i_{1}i_2\cdots i_{m}}\geq 0.$
$\mathcal{A}=(a_{i_{1}\cdots i_{m}})\in \mathbb{R}^{[m,n]}$ is called symmetric \cite{qi2005} if
\begin{eqnarray}
a_{i_{1}\cdots i_{m}}=a_{\pi(i_{1}\cdots i_{m})},\ \forall \pi\in\Pi_{m},\nonumber
\end{eqnarray}
where $\Pi_{m}$ is the permutation group of $m$ indices.
$\mathcal{A}=(a_{i_1i_2\cdots i_{m}})\in \mathbb{R}^{[m,n]}$ is called weakly symmetric \cite{changkc1} if the associated homogeneous polynomial
$$\mathcal{A}x^m=\sum\limits_{i_1,i_2,\cdots,i_m\in N}a_{i_{1}i_2\cdots i_{m}}x_{i_1}x_{i_2}\cdots x_{i_m}$$
satisfies $\nabla \mathcal{A}x^m=m\mathcal{A}x^{m-1}$.
It is shown in \cite{changkc1} that a symmetric tensor is necessarily weakly symmetric, but the converse
is not true in general.

Given a tensor $\mathcal {A}=(a_{i_1\cdots i_m})\in \mathbb{R}^{[m,n]}$, if there are $\lambda\in \mathbb{C}$ and
$x=(x_1,x_{2}\cdots,x_n)^T\in \mathbb{C}^n\backslash\{0\}$
such that
\[\mathcal {A}x^{m-1}=\lambda x~\textmd{and}~ x^Tx=1,\]
then $\lambda$ is called an $E$-eigenvalue of $\mathcal {A}$ and $x$ an
$E$-eigenvector of $\mathcal {A}$ associated with $\lambda$, where
$\mathcal {A}x^{m-1}$ is an $n$ dimension vector whose $i$th component is
\[(\mathcal
{A}x^{m-1})_i=\sum\limits_{i_2,\cdots,i_m\in N} a_{ii_2\cdots
i_m}x_{i_2}\cdots x_{i_m}.\]
If $\lambda$ and $x$ are all real, then $\lambda$ is called a $Z$-eigenvalue of $\mathcal
{A}$ and $x$ a $Z$-eigenvector of $\mathcal {A}$ associated with
$\lambda$; for details, see \cite{qi2005,lim}.
Here, we define the $Z$-spectrum of $\mathcal{A}$,
denoted $\sigma(\mathcal{A})$ to be the set of all $Z$-eigenvalues of $\mathcal{A}$. Assume $\sigma(\mathcal{A})\neq 0,$ then the $Z$-spectral radius \cite{changkc1} of $\mathcal{A}$, denoted $\varrho(\mathcal{A})$, is defined as
$$\varrho(\mathcal{A}):=\sup\{|\lambda|:\lambda\in \sigma(\mathcal{A})\}.$$

Recently, many people have focused on locating all $Z$-eigenvalues of tensors and
bounding the $Z$-spectral radius of nonnegative tensors in \cite{changkc1,wg,zjxjia,sys,liwen,hejunjcaa,hejunspringerplus,hjaml,lql}.
In 2017, Wang \emph{et al}. \cite{wg} established the following Ger$\breve{s}$gorin-type  $Z$-eigenvalue inclusion theorem for tensors.
\begin{thm}\emph{\cite[Theorem 3.1]{wg}}\label{wg-th1}
Let $\mathcal{A}=(a_{i_{1}\cdots i_{m}})\in{\mathbb{R}}^{[m,n]}$. Then
\begin{eqnarray*}
\sigma(\mathcal{A})\subseteq \mathcal{K}(\mathcal{A})=\bigcup\limits_{i\in{N}}\mathcal{K}_{i}(\mathcal{A}),
\end{eqnarray*}
where 
\begin{eqnarray*}
\mathcal{K}_{i}(\mathcal{A})=\{z\in{\mathbb{C}}:|z|\leq R_{i}(\mathcal{A})\},~R_{i}(\mathcal{A})=\sum\limits_{i_2,\cdots, i_m\in N}|a_{ii_{2}\cdots i_{m}}|.
\end{eqnarray*}
\end{thm}

To get a tighter $Z$-eigenvalue inclusion set than $\mathcal{K}(\mathcal{A})$,
Wang \emph{et al}. \cite{wg} gave the following Brauer-type $Z$-eigenvalue localization set for tensors.

\begin{thm}\emph{\cite[Theorem 3.2]{wg}}\label{wg-th2}
Let $\mathcal{A}=(a_{i_{1}\cdots i_{m}})\in{\mathbb{R}}^{[m,n]}$. Then
\begin{eqnarray*}
\sigma(\mathcal{A})\subseteq\mathcal{L}(\mathcal{A})
=\bigcup\limits_{i\in N}\bigcap\limits_{j\in N,j\neq i}\mathcal{L}_{i,j}(\mathcal{A}),
\end{eqnarray*}
where
\[
\mathcal{L}_{i,j}(\mathcal{A})=\left\{z\in{\mathbb{C}}:\big(|z|-(R_i(\mathcal{A})-|a_{ij\cdots j}|\big)|z|\leq |a_{ij\cdots j}|R_j(\mathcal{A})\right\}.
\]
\end{thm}

Very recently, Zhao \cite{zjxjia} presented another Brauer-type $Z$-eigenvalue localization set for tensors and proved that this set is tighter than those in Theorem \ref{wg-th1} and Theorem \ref{wg-th2}.

\begin{thm}\emph{\cite[Theorem 3]{zjxjia}}\label{zjxjia}
Let $\mathcal{A}=(a_{i_{1}\cdots i_{m}})\in{\mathbb{R}}^{[m,n]}$. Then
\begin{eqnarray*}
\sigma(\mathcal{A})\subseteq \Psi(\mathcal{A})=\bigcup\limits_{i\in N}\bigcap\limits_{j\in N, j\neq i}\Psi_{i,j}(\mathcal{A}),
\end{eqnarray*}
where
\begin{eqnarray*}
\Psi_{i,j}(\mathcal{A})=\left\{z\in \mathbb{C}:\big(|z|-r_i^{\overline{\Delta}_j}(\mathcal{A})\big)|z|\leq r_i^{\Delta_j}(\mathcal{A})R_j(\mathcal{A})\right\},
\end{eqnarray*}
\begin{eqnarray*}
r_i^{\Delta_j}(\mathcal{A})=\sum\limits_{j\in \{i_2,\cdots,i_m\}}|a_{ii_2\cdots i_m}|,
~~r_i^{\overline{\Delta}_j}(\mathcal{A})=\sum\limits_{j\notin\{i_2,\cdots,i_m\}}|a_{ii_2\cdots i_m}|.
\end{eqnarray*}
\end{thm}

As we know, one can use eigenvalue inclusion sets to obtain the upper bound of the spectral radius of nonnegative tensors; for details, see \cite{wg,lcq-lyt,lcq-kx,lcq-zjj,lcq-cz}. Therefore, the main aim of this paper is to give a new $Z$-eigenvalue inclusion set for tensors and prove that the new set is tighter than those in
Theorems \ref{wg-th1}-\ref{zjxjia}. And as an application, a new upper bound for the $Z$-spectral radius of weakly symmetric nonnegative tensors is obtained and proved to be sharper than some existing upper bounds.

\section{Main results}\label{sec2}
In this section, we give a new Brauer-type $Z$-eigenvalue localization set for tensors, and establish the comparison between the new set with those in Theorems \ref{wg-th1}-\ref{zjxjia}.

\begin{thm}\label{th1}
Let $\mathcal{A}=(a_{i_{1}\cdots i_{m}})\in{\mathbb{R}}^{[m,n]}$. Then
\begin{eqnarray*}
\sigma(\mathcal{A})\subseteq \Omega(\mathcal{A})=
\left(\bigcup\limits_{i\in N}\bigcap\limits_{j\in N, j\neq i}\hat{\Omega}_{i,j}(\mathcal{A})\right)
\bigcup
\left(\bigcup\limits_{i\in N}\bigcap\limits_{j\in N,j\neq i}\Big(\tilde{\Omega}_{i,j}(\mathcal{A})\bigcap \mathcal{K}_i(\mathcal{A})\Big)\right),
\end{eqnarray*}
where
\begin{eqnarray*}
\hat{\Omega}_{i,j}(\mathcal{A})=\left\{z\in \mathbb{C}:
|z|< r_i^{\overline{\Delta}_j}(\mathcal{A}),|z|<r_j^{\Delta_j}(\mathcal{A})\right\}
\end{eqnarray*}
and
\begin{eqnarray*}
\tilde{\Omega}_{i,j}(\mathcal{A})=\left\{z\in \mathbb{C}:
\big(|z|-r_i^{\overline{\Delta}_j}(\mathcal{A})\big)\big(|z|-r_j^{\Delta_j}(\mathcal{A})\big)
\leq r_i^{\Delta_j}(\mathcal{A})r_j^{\overline{\Delta}_j}(\mathcal{A})
\right\}.
\end{eqnarray*}
\end{thm}
\begin{proof}
Let $\lambda$ be a $Z$-eigenvalue of $\mathcal{A}$ with corresponding $Z$-eigenvector $x=(x_{1},\cdots,x_{n})^{T}\in{\mathbb{C}}^{n}\backslash\{0\}$, i.e.,
\begin{eqnarray}\label{th1-equ1}
\mathcal{A}x^{m-1}=\lambda x,~\textmd{and}~||x||_2=1.
\end{eqnarray}
Let $|x_t|=\max\limits_{i \in N}|x_i|$. Obviously, $0<|x_t|^{m-1}\leq |x_t|\leq 1.$
For $\forall~j\in N, j\neq t$, from (\ref{th1-equ1}), we have
\begin{eqnarray*}
\lambda x_t=\sum\limits_{j\in \{i_2,\cdots, i_m\}}a_{ti_{2}\cdots i_{m}}x_{i_{2}}\cdots x_{i_{m}}
+\sum\limits_{j\notin \{i_2,\cdots, i_m\}}a_{ti_{2}\cdots i_{m}}x_{i_{2}}\cdots x_{i_{m}}.
\end{eqnarray*}
Taking modulus in the above equation and using the triangle inequality gives
\begin{eqnarray*}
|\lambda||x_t|&\leq& \sum\limits_{j\in \{i_2,\cdots, i_m\}}|a_{ti_{2}\cdots i_{m}}||x_{i_{2}}|\cdots |x_{i_{m}}|
+\sum\limits_{j\notin \{i_2,\cdots, i_m\}}|a_{ti_{2}\cdots i_{m}}||x_{i_{2}}|\cdots |x_{i_{m}}|\\
&\leq& \sum\limits_{j\in \{i_2,\cdots, i_m\}}|a_{ti_{2}\cdots i_{m}}||x_j|
+\sum\limits_{j\notin \{i_2,\cdots, i_m\}}|a_{ti_{2}\cdots i_{m}}||x_t|\\
&=&r_t^{\Delta_j}(\mathcal{A})|x_j|+r_t^{\overline{\Delta}_j}(\mathcal{A})|x_t|,
\end{eqnarray*}
i.e.,
\begin{eqnarray}\label{th1-equ2}
\big(|\lambda|-r_t^{\overline{\Delta}_j}(\mathcal{A})\big)|x_t|\leq r_t^{\Delta_j}(\mathcal{A})|x_j|.
\end{eqnarray}
If $|x_j|=0$, then $|\lambda|-r_t^{\overline{\Delta}_j}(\mathcal{A})\leq 0$ as $|x_t|>0$.
When $|z|-r_j^{\Delta_j}(\mathcal{A})\geq 0$, we have
\begin{eqnarray*}
\big(|\lambda|-r_t^{\overline{\Delta}_j}(\mathcal{A})\big)\big(|\lambda|-r_j^{\Delta_j}(\mathcal{A})\big)
\leq 0\leq r_t^{\Delta_j}(\mathcal{A})r_j^{\overline{\Delta}_j}(\mathcal{A}),
\end{eqnarray*}
which implies $\lambda\in \bigcap\limits_{j\in N,j\neq t}\tilde{\Omega}_{t,j}(\mathcal{A})\subseteq \Omega(\mathcal{A})$ from the arbitrariness of $j$.
When $|z|-r_j^{\Delta_j}(\mathcal{A})<0$, from the arbitrariness of $j$, we have $\lambda\in \bigcap\limits_{j\in N,j\neq t}\hat{\Omega}_{t,j}(\mathcal{A})\subseteq \Omega(\mathcal{A})$.

Otherwise, $|x_j|>0$. From (\ref{th1-equ1}), we can get
\begin{eqnarray*}
|\lambda||x_j|&\leq& \sum\limits_{j\in \{i_2,\cdots, i_m\}}|a_{ji_{2}\cdots i_{m}}||x_{i_{2}}|\cdots |x_{i_{m}}|
+\sum\limits_{j\notin \{i_2,\cdots, i_m\}}|a_{ji_{2}\cdots i_{m}}||x_{i_{2}}|\cdots |x_{i_{m}}|\\
&\leq& \sum\limits_{j\in \{i_2,\cdots, i_m\}}|a_{ji_{2}\cdots i_{m}}||x_j|
+\sum\limits_{j\notin \{i_2,\cdots, i_m\}}|a_{ji_{2}\cdots i_{m}}||x_t|\\
&=&r_j^{\Delta_j}(\mathcal{A})|x_j|+r_j^{\overline{\Delta}_j}(\mathcal{A})|x_t|,
\end{eqnarray*}
i.e.,
\begin{eqnarray}\label{th1-equ3}
\big(|\lambda|-r_j^{\Delta_j}(\mathcal{A})\big)|x_j|\leq r_j^{\overline{\Delta}_j}(\mathcal{A})|x_t|.
\end{eqnarray}
By (\ref{th1-equ2}), it is not difficult to see $\lambda\in \mathcal{K}_t(\mathcal{A})$.
When $|\lambda|-r_t^{\overline{\Delta}_j}(\mathcal{A})\geq 0$ or $|\lambda|-r_j^{\Delta_j}(\mathcal{A})\geq 0$ holds,
multiplying (\ref{th1-equ2}) with (\ref{th1-equ3}) and noting that $|x_t||x_j|>0$, we have
\begin{eqnarray*}
\big(|\lambda|-r_t^{\overline{\Delta}_j}(\mathcal{A})\big)\big(|\lambda|-r_j^{\Delta_j}(\mathcal{A})\big)
\leq r_t^{\Delta_j}(\mathcal{A})r_j^{\overline{\Delta}_j}(\mathcal{A}),
\end{eqnarray*}
which implies
$\lambda\in \bigcap\limits_{j\in N,j\neq t}\big(\tilde{\Omega}_{t,j}(\mathcal{A})\bigcap \mathcal{K}_t(\mathcal{A})\big)\subseteq \Omega(\mathcal{A})$
from the arbitrariness of $j$.
And when $|\lambda|-r_t^{\overline{\Delta}_j}(\mathcal{A})< 0$ and $|\lambda|-r_j^{\Delta_j}(\mathcal{A})< 0$ hold,
from the arbitrariness of $j$,
we have $\lambda\in \bigcap\limits_{j\in N,j\neq t}\hat{\Omega}_{i,j}(\mathcal{A})\subseteq \Omega(\mathcal{A})$.
Hence, the conclusion $\sigma(\mathcal{A})\subseteq \Omega(\mathcal{A})$ follows immediately from what we have proved.
\end{proof}

Next, a comparison theorem is given for Theorems \ref{wg-th1}-\ref{th1}.
\begin{thm}\label{th2}
Let $\mathcal{A}=(a_{i_{1}\cdots i_{m}})\in{\mathbb{R}}^{[m,n]}$. Then
\begin{eqnarray*}
\Omega(\mathcal{A})\subseteq\Psi(\mathcal{A})\subseteq \mathcal{L}(\mathcal{A})\subseteq \mathcal{K}(\mathcal{A}).
\end{eqnarray*}
\end{thm}
\begin{proof}
From Theorem 5 in \cite{zjxjia}, we have $\Psi(\mathcal{A})\subseteq\mathcal{L}(\mathcal{A})\subseteq \mathcal{K}(\mathcal{A})$.
Hence, here only $\Omega(\mathcal{A})\subseteq\Psi(\mathcal{A})$ is proved.
Let $z\in \Omega(\mathcal{A})$. Then
$z\in \bigcup\limits_{i\in N}\bigcap\limits_{j\in N, j\neq i}\hat{\Omega}_{i,j}(\mathcal{A})$
or
$z\in \bigcup\limits_{i\in N}\bigcap\limits_{j\in N,j\neq i}\Big(\tilde{\Omega}_{i,j}(\mathcal{A})\bigcap \mathcal{K}_i(\mathcal{A})\Big).$
We next divide the proof into two cases.

Case I: If $z\in \bigcup\limits_{i\in N}\bigcap\limits_{j\in N, j\neq i}\hat{\Omega}_{i,j}(\mathcal{A}),$
then there is one index $i\in N$ such that
$|z|< r_i^{\overline{\Delta}_j}(\mathcal{A})$
and
$|z|<r_j^{\Delta_j}(\mathcal{A}), \forall~j\in N, j\neq i.$
Then, it is easy to see that
$$\big(|z|-r_i^{\overline{\Delta}_j}(\mathcal{A})\big)|z|\leq 0\leq r_i^{\Delta_j}(\mathcal{A})R_j(\mathcal{A}),~\forall~j\in N, j\neq i,$$
which implies that
$z\in \bigcap\limits_{j\in N, j\neq i}\Psi_{i,j}(\mathcal{A})\subseteq \Psi(\mathcal{A})$.
This implies $\Omega(\mathcal{A})\subseteq\Psi(\mathcal{A})$.

Case II: If $z\in \bigcup\limits_{i\in N}\bigcap\limits_{j\in N,j\neq i}\Big(\tilde{\Omega}_{i,j}(\mathcal{A})\bigcap \mathcal{K}_i(\mathcal{A})\Big),$
then there is one index $i\in N$, for any $j\in N, j\neq i$, such that
\begin{eqnarray}\label{th2-equ1}
|z|\leq R_i(\mathcal{A}),
\end{eqnarray}
and
\begin{eqnarray}\label{th2-equ2}
\big(|z|-r_i^{\overline{\Delta}_j}(\mathcal{A})\big)\big(|z|-r_j^{\Delta_j}(\mathcal{A})\big)
\leq r_i^{\Delta_j}(\mathcal{A})r_j^{\overline{\Delta}_j}(\mathcal{A}).
\end{eqnarray}

(i) If $r_i^{\Delta_j}(\mathcal{A})r_j^{\overline{\Delta}_j}(\mathcal{A})=0$,
then $|z|\leq r_i^{\overline{\Delta}_j}(\mathcal{A})$ or $|z|\leq r_j^{\Delta_j}(\mathcal{A}).$
When $|z|\leq r_i^{\overline{\Delta}_j}(\mathcal{A})$, we have
\begin{eqnarray*}
\big(|z|-r_i^{\overline{\Delta}_j}(\mathcal{A})\big)|z|\leq 0\leq r_i^{\Delta_j}(\mathcal{A})R_j(\mathcal{A}),
\end{eqnarray*}
which implies that
$z\in \bigcap\limits_{j\in N, j\neq i}\Psi_{i,j}(\mathcal{A})\subseteq \Psi(\mathcal{A})$
from the arbitrariness of $j$.
When $|z|\leq r_j^{\Delta_j}(\mathcal{A}),$ we have
\begin{eqnarray}\label{th2-equ3}
|z|\leq R_j(\mathcal{A}).
\end{eqnarray}
From (\ref{th2-equ1}), we can get
\begin{eqnarray}\label{th2-equ4}
|z|-r_i^{\overline{\Delta}_j}(\mathcal{A})\leq r_i^{\Delta_j}(\mathcal{A}).
\end{eqnarray}
Multiplying (\ref{th2-equ3}) and (\ref{th2-equ4}), we have
\begin{eqnarray}\label{th2-equ5}
\big(|z|-r_i^{\overline{\Delta}_j}(\mathcal{A})\big)|z|
\leq r_i^{\Delta_j}(\mathcal{A})R_j(\mathcal{A}),
\end{eqnarray}
which also implies that $z\in \bigcap\limits_{j\in N, j\neq i}\Psi_{i,j}(\mathcal{A})\subseteq \Psi(\mathcal{A})$, consequently,
$\Omega(\mathcal{A})\subseteq \Psi(\mathcal{A}).$

(ii) If $r_i^{\Delta_j}(\mathcal{A})r_j^{\overline{\Delta}_j}(\mathcal{A})>0$,
then dividing both sides by $r_i^{\Delta_j}(\mathcal{A})r_j^{\overline{\Delta}_j}(\mathcal{A})$ in (\ref{th2-equ2}), we have
\begin{eqnarray}\label{th2-equ6}
\frac{|z|-r_i^{\overline{\Delta}_j}(\mathcal{A})}{r_i^{\Delta_j}(\mathcal{A})}
\frac{|z|-r_j^{\Delta_j}(\mathcal{A})}{r_j^{\overline{\Delta}_j}(\mathcal{A})}\leq 1.
\end{eqnarray}
From (\ref{th2-equ1}), we can get (\ref{th2-equ4}) and furthermore $\frac{|z|-r_i^{\overline{\Delta}_j}(\mathcal{A})}{r_i^{\Delta_j}(\mathcal{A})}\leq 1.$
When $\frac{|z|-r_j^{\Delta_j}(\mathcal{A})}{r_j^{\overline{\Delta}_j}(\mathcal{A})}\leq 1$, then (\ref{th2-equ3}) holds.
Multiplying (\ref{th2-equ3}) and (\ref{th2-equ4}), we can get (\ref{th2-equ5}),
which implies that $z\in \bigcap\limits_{j\in N, j\neq i}\Psi_{i,j}(\mathcal{A})\subseteq \Psi(\mathcal{A})$, consequently,
$\Omega(\mathcal{A})\subseteq \Psi(\mathcal{A}).$

And when $\frac{|z|-r_j^{\Delta_j}(\mathcal{A})}{r_j^{\overline{\Delta}_j}(\mathcal{A})}>1$, we can obtain $|z|>R_j(\mathcal{A}).$
Let $a=|z|,b=r_j^{\Delta_j}(\mathcal{A})-|a_{jj\cdots j}|,c=|a_{jj\cdots j}|$ and $d=r_j^{\overline{\Delta}_j}(\mathcal{A})$.
By Lemma 2.3 in \cite{lcq-lyt}, we have
\begin{eqnarray}\label{th2-equ7}
\frac{|z|}{R_j(\mathcal{A})}=\frac{a}{b+c+d}\leq\frac{a-(b+c)}{d}=\frac{|z|-r_j^{\Delta_j}(\mathcal{A})}{r_j^{\overline{\Delta}_j}(\mathcal{A})}.
\end{eqnarray}
If $|z|> r_i^{\overline{\Delta}_j}(\mathcal{A})$, by (\ref{th2-equ6}) and (\ref{th2-equ7}), we have
\begin{eqnarray*}
\frac{|z|-r_i^{\overline{\Delta}_j}(\mathcal{A})}{r_i^{\Delta_j}(\mathcal{A})}
\frac{|z|}{R_j(\mathcal{A})}\leq
\frac{|z|-r_i^{\overline{\Delta}_j}(\mathcal{A})}{r_i^{\Delta_j}(\mathcal{A})}
\frac{|z|-r_j^{\Delta_j}(\mathcal{A})}{r_j^{\overline{\Delta}_j}(\mathcal{A})}\leq 1,
\end{eqnarray*}
equivalently,
\begin{eqnarray*}
\big(|z|-r_i^{\overline{\Delta}_j}(\mathcal{A})\big)|z|
\leq r_i^{\Delta_j}(\mathcal{A})R_j(\mathcal{A}),
\end{eqnarray*}
which implies that $z\in \bigcap\limits_{j\in N, j\neq i}\Psi_{i,j}(\mathcal{A})\subseteq \Psi(\mathcal{A})$ from the arbitrariness of $j$.
If $|z|\leq r_i^{\overline{\Delta}_j}(\mathcal{A})$, we have
\begin{eqnarray*}
\big(|z|-r_i^{\overline{\Delta}_j}(\mathcal{A})\big)|z|\leq 0\leq r_i^{\Delta_j}(\mathcal{A})R_j(\mathcal{A}).
\end{eqnarray*}
This also leads to
$z\in \bigcap\limits_{j\in N, j\neq i}\Psi_{i,j}(\mathcal{A})\subseteq \Psi(\mathcal{A})$,
consequently,
$\Omega(\mathcal{A})\subseteq \Psi(\mathcal{A}).$
The conclusion follows from Case I and Case II.
\end{proof}

\begin{remark}\label{remark1}\emph{
Theorem \ref{th2} shows that the set $\Omega(\mathcal{A})$ in Theorem \ref{th1} is tighter than $\mathcal{K}(\mathcal{A})$ in Theorem \ref{wg-th1},
$\mathcal{L}(\mathcal{A})$ in Theorem \ref{wg-th2} and $\Psi(\mathcal{A})$ in Theorem \ref{zjxjia}, that is,
$\Omega(\mathcal{A})$ can capture all $Z$-eigenvalues of $\mathcal{A}$ more precisely than $\mathcal{K}(\mathcal{A})$, $\mathcal{L}(\mathcal{A})$
and $\Psi(\mathcal{A})$.}
\end{remark}

Now, an example is given to verify the fact in Remark \ref{remark1}.

\begin{example}\label{eg1}
Let $\mathcal{A}=(a_{ijkl})\in{\mathbb{R}}^{[4,2]}$ be a symmetric tensor defined by
$$a_{1111}=1,~a_{1112}=1,a_{1122}=0.25,~a_{2222}=5,~and~a_{ijkl}=0~elsewhere.$$
By computation, we get that all the $Z$-eigenvalues of $\mathcal{A}$ are $-0.2044,-0.2044,5.0000$ and $5.0000$.
By Theorem \ref{wg-th1}, we have
\begin{eqnarray*}
\mathcal{K}(\mathcal{A})=\{z\in{\mathbb{C}}:|z|\leq 6.7500\}.
\end{eqnarray*}
By Theorem \ref{wg-th2}, we have
\begin{eqnarray*}
\mathcal{L}(\mathcal{A})=\{z\in{\mathbb{C}}:|z|\leq 6.4827\}.
\end{eqnarray*}
By Theorem \ref{zjxjia}, we have
\begin{eqnarray*}
\Psi(\mathcal{A})=\{z\in{\mathbb{C}}:|z|\leq 6.3161\}.
\end{eqnarray*}
By Theorem \ref{th1}, we have
\begin{eqnarray*}
\Omega(\mathcal{A})=\{z\in{\mathbb{C}}:|z|\leq 5.0000\}.
\end{eqnarray*}
The $Z$-eigenvalue inclusion sets $\mathcal{K}(\mathcal{A})$, $\mathcal{L}(\mathcal{A})$, $\Psi(\mathcal{A})$, $\Omega(\mathcal{A})$
and the exact $Z$-eigenvalues are drawn in Figure 1, where
$\mathcal{K}(\mathcal{A})$, $\mathcal{L}(\mathcal{A})$, $\Psi(\mathcal{A})$ and $\Omega(\mathcal{A})$ are represented by black dashed boundary, green solid boundary, blue point line boundary and red solid boundary, respectively.
The exact eigenvalues are plotted by black ``$+$".
It is easy to see $\sigma(\mathcal{A})\subseteq \Omega(\mathcal{A})\subset \Psi(\mathcal{A})\subset \mathcal{L}(\mathcal{A})\subset \mathcal{K}(\mathcal{A})$,
that is,
$\Omega(\mathcal{A})$ can capture all $Z$-eigenvalues of $\mathcal{A}$ more precisely than $\Psi(\mathcal{A})$, $\mathcal{L}(\mathcal{A})$ and $\mathcal{K}(\mathcal{A})$.
\begin{figure}[!ht]
\centerline{\includegraphics[height=10cm,width=20cm]{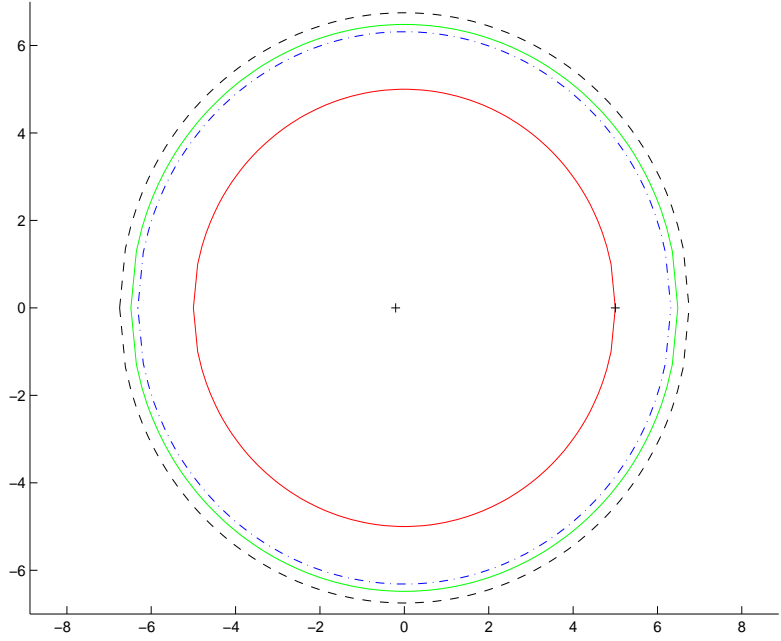}}
{\caption{Comparisons of $\mathcal{K}(\mathcal{A})$, $\mathcal{L}(\mathcal{A})$, $\Psi(\mathcal{A})$ and $\Omega(\mathcal{A})$.}}
\label{Fig1}
\end{figure}
\end{example}


\section{A sharper upper bound for the $Z$-spectral radius of weakly symmetric nonnegative tensors}
As the $Z$-spectral radius of weakly symmetric nonnegative tensors plays a fundamental role in the symmetric best rank-one approximation \cite{lql,ko},
recently, many people focus on bounding the $Z$-spectral radius of weakly symmetric nonnegative tensors.
As an application of the set in Theorem \ref{th1}, we in this section give a sharper upper bound for the $Z$-spectral radius of weakly symmetric nonnegative tensors.

\begin{thm}\label{th3}
Let $\mathcal{A}=(a_{i_{1}\cdots i_{m}})\in{\mathbb{R}}^{[m,n]}$ be a weakly symmetric nonnegative tensor. Then
\begin{eqnarray*}
\varrho(\mathcal{A})\leq \Omega_{max}(\mathcal{A})=\max\big\{\hat{\Omega}_{max}(\mathcal{A}), \tilde{\Omega}_{max}(\mathcal{A})\big\},
\end{eqnarray*}
where
\begin{eqnarray*}
&&\hat{\Omega}_{max}(\mathcal{A})=\max\limits_{i\in N}\min\limits_{j\in N, j\neq i}\min\big\{r_i^{\overline{\Delta}_j}(\mathcal{A}),r_j^{\Delta_j}(\mathcal{A})\big\},\\
&&\tilde{\Omega}_{max}(\mathcal{A})=\max\limits_{i\in N}\min\limits_{j\in N,j\neq i}\min\left\{R_i(\mathcal{A}),\bar{\Omega}_{i,j}(\mathcal{A})\right\},
\end{eqnarray*}
and
\begin{eqnarray*}
\bar{\Omega}_{i,j}(\mathcal{A})=\frac{1}{2}\left\{
r_i^{\overline{\Delta}_j}(\mathcal{A})+r_j^{\Delta_j}(\mathcal{A})+
\sqrt{\big(r_i^{\overline{\Delta}_j}(\mathcal{A})-r_j^{\Delta_j}(\mathcal{A})\big)^2+4r_i^{\Delta_j}(\mathcal{A})r_j^{\overline{\Delta}_j}(\mathcal{A})}
\right\}.
\end{eqnarray*}
\end{thm}
\begin{proof}
By Lemma 4.4 in \cite{wg}, we know that $\varrho(\mathcal{A})$ is a $Z$-eigenvalue of $\mathcal{A}$.
By Theorem \ref{th1}, we have
$$\varrho(\mathcal{A})\in \bigcup\limits_{i\in N}\bigcap\limits_{j\in N, j\neq i}\hat{\Omega}_{i,j}(\mathcal{A})
~\textmd{or}~
\varrho(\mathcal{A})\in\bigcup\limits_{i\in N}\bigcap\limits_{j\in N, j\neq i}\Big(\tilde{\Omega}_{i,j}(\mathcal{A})\bigcap \mathcal{K}_i(\mathcal{A})\Big).$$
If $\varrho(\mathcal{A})\in \bigcup\limits_{i\in N}\bigcap\limits_{j\in N, j\neq i}\hat{\Omega}_{i,j}(\mathcal{A}),$
then there is one index $i\in N$ such that
$$\varrho(\mathcal{A})< r_i^{\overline{\Delta}_j}(\mathcal{A})
~\textmd{and}~
\varrho(\mathcal{A})<r_j^{\Delta_j}(\mathcal{A}), \forall~j\in N, j\neq i.$$
Then we have
$\varrho(\mathcal{A})\leq \min\limits_{j\in N, j\neq i}\min\big\{r_i^{\overline{\Delta}_j}(\mathcal{A}),r_j^{\Delta_j}(\mathcal{A})\big\}.$
Furthermore, we have
\begin{eqnarray*}
\varrho(\mathcal{A})\leq \max\limits_{i\in N}\min\limits_{j\in N, j\neq i}\min\big\{r_i^{\overline{\Delta}_j}(\mathcal{A}),r_j^{\Delta_j}(\mathcal{A})\big\}.
\end{eqnarray*}

If
$\varrho(\mathcal{A})\in \bigcup\limits_{i\in N}\bigcap\limits_{j\in N,j\neq i}\Big(\tilde{\Omega}_{i,j}(\mathcal{A})\bigcap \mathcal{K}_i(\mathcal{A})\Big),$
then there is one index $i\in N$, for any $j\in N, j\neq i$, such that
\begin{eqnarray}\label{th3-equ1}
\varrho(\mathcal{A})\leq R_i(\mathcal{A})
\end{eqnarray}
and
\begin{eqnarray*}
\big(\varrho(\mathcal{A})-r_i^{\overline{\Delta}_j}(\mathcal{A})\big)\big(\varrho(\mathcal{A})-r_j^{\Delta_j}(\mathcal{A})\big)
\leq r_i^{\Delta_j}(\mathcal{A})r_j^{\overline{\Delta}_j}(\mathcal{A}).
\end{eqnarray*}
Solving $\varrho(\mathcal{A})$ in above inequality gives
\begin{eqnarray}\label{th3-equ2}
\varrho(\mathcal{A})\leq\frac{1}{2}\left\{
r_i^{\overline{\Delta}_j}(\mathcal{A})+r_j^{\Delta_j}(\mathcal{A})+
\sqrt{\big(r_i^{\overline{\Delta}_j}(\mathcal{A})-r_j^{\Delta_j}(\mathcal{A})\big)^2+4r_i^{\Delta_j}(\mathcal{A})r_j^{\overline{\Delta}_j}(\mathcal{A})}
\right\}=\bar{\Omega}_{i,j}(\mathcal{A}).
\end{eqnarray}
Combining (\ref{th3-equ1}) and (\ref{th3-equ2}), and by the arbitrariness of $j$, we have
\begin{eqnarray*}
\varrho(\mathcal{A})\leq\min\limits_{j\in N, j\neq i}\min\left\{R_i(\mathcal{A}),\bar{\Omega}_{i,j}(\mathcal{A})\right\}
\leq\max\limits_{i\in N}\min\limits_{j\in N,j\neq i}\min\left\{R_i(\mathcal{A}),\bar{\Omega}_{i,j}(\mathcal{A})\right\}.
\end{eqnarray*}
The conclusion follows from what we have proved.
\end{proof}

By Corollary 4.1 of \cite{wg}, Theorem 6 of \cite{zjxjia} and Theorem \ref{th2}, the following comparison theorem can be derived easily.

\begin{thm}\label{th4}
Let $\mathcal{A}=(a_{i_{1}\cdots i_{m}})\in{\mathbb{R}}^{[m,n]}$ be a weakly symmetric nonnegative tensor. Then
the upper bound in Theorem \ref{th3} is smaller than those in Theorem 5 of \cite{zjxjia}, Theorem 4.5 of \cite{wg} and Corollary 4.5 of \cite{sys},
that is,
\begin{eqnarray*}
\Omega_{max}(\mathcal{A})&\leq&\max\limits_{i\in N}\min\limits_{j\in N, j\neq i}\frac{1}{2}\left\{
r_i^{\overline{\Delta}_j}(\mathcal{A})+\sqrt{(r_i^{\overline{\Delta}_j}(\mathcal{A}))^2+4r_i^{\Delta_j}(\mathcal{A})R_j(\mathcal{A})}\right\}\\
&\leq&\max\limits_{i\in N}\min\limits_{j\in N, j\neq i}\frac{1}{2}\left\{
R_i(\mathcal{A})-a_{ij\cdots j}+\sqrt{(R_i(\mathcal{A})-a_{ij\cdots j})^2+4a_{ij\cdots j}R_j(\mathcal{A})}\right\}\\
&\leq&\max\limits_{i\in N}R_i(\mathcal{A}).
\end{eqnarray*}
\end{thm}

Finally, we show that the upper bound in Theorem \ref{th3} is smaller than those in \cite{wg,zjxjia,sys,liwen,hejunjcaa,hejunspringerplus,lql} by the following example.

\begin{example}\label{eg2}
\setlength\arraycolsep{4pt}\emph{
Let $\mathcal{A}=(a_{ijk})\in {\mathbb{R}}^{[3,3]}$ be a weakly symmetric nonnegative tensor with entries defined as follows:
$$\mathcal{A}(:,:,1)=\left(
\begin{array}{ccc}
3&3&0\\
3&2&2.5\\
0.5&2.5&0\\
\end{array}\right),
~\mathcal{A}(:,:,2)=\left(
\begin{array}{ccc}
3&2&2\\
2&0&3\\
2.5&3&1\\
\end{array}\right),
~\mathcal{A}(:,:,3)=\left(
\begin{array}{ccc}
1&3&0\\
2.5&3&1\\
0&1&0\\
\end{array}\right).
$$
By Corollary 4.5 of \cite{sys} and Theorem 3.3 of \cite{liwen}, we both have $$\varrho(\mathcal{A})\leq 19.$$
By Theorem 3.5 of \cite{hejunjcaa}, we have $$\varrho(\mathcal{A})\leq 18.6788.$$
By Theorem 4.6 of \cite{wg}, we have $$\varrho(\mathcal{A})\leq 18.6603.$$
By Theorem 4.5 of \cite{wg} and Theorem 6 of \cite{hejunspringerplus}, we both have $$\varrho(\mathcal{A})\leq 18.5656.$$
By Theorem 4.7 of \cite{wg}, we have $$\varrho(\mathcal{A})\leq 18.3417.$$
By Theorem 2.9 of \cite{lql}, we have $$\varrho(\mathcal{A})\leq 17.2063.$$
By Theorem 5 of \cite{zjxjia}, we obtain $$\varrho(\mathcal{A})\leq 15.2580,$$
By Theorem \ref{th3}, we obtain $$\varrho(\mathcal{A})\leq 14.9410.$$
This example shows that the bound in Theorem \ref{th3} is the smallest.}
\end{example}

\begin{remark}\label{remark1}\emph{
From Example \ref{eg1}, it is not difficult to see that the upper bound in Theorem \ref{th3} could reach the true value of $\varrho(\mathcal{A})$ in some cases.}
\end{remark}


\section{Conclusion}\label{Sec5}

In this paper, we present a new $Z$-eigenvalue localization set $\Omega(\mathcal{A})$ and prove that this set is tighter than those in \cite{wg,zjxjia}.
As an application, we obtain a new upper bound $\Omega_{max}(\mathcal{A})$ for the $Z$-spectral radius of weakly symmetric nonnegative tensors,
and show that this bound is sharper than those in \cite{wg,zjxjia,sys,liwen,hejunjcaa,hejunspringerplus,lql} in some cases by a numerical example.

\section*{Acknowledgments}
This work is supported by National Natural Science Foundations of China (Grant No.11501141),
Foundation of Guizhou Science and Technology Department (Grant No.[2015]2073)
and Natural Science Programs of Education Department of Guizhou Province (Grant No.[2016]066).

\enddocument